\documentclass[10pt, reqno]{amsart}
\usepackage{amsmath}
\usepackage{amscd}
\usepackage{amsthm}
\usepackage{amssymb} \usepackage{latexsym}
\usepackage{euscript}
\usepackage{epsfig}
\usepackage{graphics}
\usepackage{array}
\usepackage{enumerate}
\usepackage{dsfont}
\usepackage{color}
\usepackage{wasysym}

\newcommand{\R}{{\mathbb R}}

\newcommand{\N}{{\mathbb N}}

\newcommand{\ph}{{\varphi}}

\newcommand{\Gm}{{\Gamma}}

\newcommand{\cE}{\mathcal{E}}

\newcommand{\dEl}{D(\mathcal{E})_{\rm loc}}

\newcommand{\Ee}{\mathcal{E}_{\rm e}}
\newcommand{\av}[1]{\lVert#1\rVert}

\newtheorem{thm}{Theorem}[section]

\newtheorem{lemma}[thm]{Lemma}

\theoremstyle{definition}
\newtheorem{definition}[thm]{Definition}

\newtheorem*{remark}{Remark}

\newcommand{\Hmm}[1]{\leavevmode{\marginpar{\tiny%
$\hbox to 0mm{\hspace*{-0.5mm}$\leftarrow$\hss}%
\vcenter{\vrule depth 0.1mm height 0.1mm width \the\marginparwidth}%
\hbox to
0mm{\hss$\rightarrow$\hspace*{-0.5mm}}$\\\relax\raggedright #1}}}

\begin{document}

\title{On $L^p$ Liouville theorems for Dirichlet forms}

\author{Bobo Hua}
\address{School of Mathematical Sciences, LMNS, Fudan University, Shanghai 200433, China; Shanghai Center for Mathematical Sciences, Jiangwan Campus, Fudan University, No. 2005 Songhu Road, Shanghai 200438, China.}
\email{bobohua@fudan.edu.cn}


\author{Matthias Keller}
\address{Institut für Mathematik, Universität Potsdam, 14476 Potsdam, Germany.} \email{matthias.keller@uni-potsdam.de}

\author{Daniel Lenz}
\address{Institut für Mathematik, Friedrich-Schiller-Universität Jena, 07743 Jena, Germany.} \email{daniel.lenz@uni-jena.de}

\author{Marcel Schmidt}
\address{Mathematisches Institut, Universität Leipzig, 04109 Leipzig, Germany.} \email{marcel.schmidt@math.uni-leipzig.de}

\begin{abstract}
	We study harmonic functions for general Dirichlet forms. First we review consequences of Fukushima's ergodic theorem for the harmonic functions in the domain of the $ L^{p} $ generator.
	Secondly we prove analogues of Yau's and  Karp's Liouville theorems for weakly harmonic functions. Both say  that weakly harmonic functions which satisfy  certain $ L^{p} $   growth criteria must be constant. As consequence we  give an integral criterion for recurrence.
\end{abstract}
\maketitle

\tableofcontents

\section{Introduction}

Liouville theorems for harmonic functions have a long tradition and circle around the idea that a harmonic function which satisfies certain boundedness conditions must be constant. The first result of this type for analytic functions was proven by Cauchy in 1844, see \cite{Cauchy}. In 1847 Liouville presented the result in a lecture  and that is how the name arose. Since then numerous results in various contexts were proven.

Our setting is the one of Dirichlet forms. This setting includes on the one hand well studied objects such as Laplacians on (sub) Riemannian manifolds but also non-local operators arising from jump processes.

We start by stating the results and refer for the definitions to Section~\ref{s:FET} and Section~\ref{s:setup}. For a general background on Dirichlet forms we refer to \cite{FOT}. Let $\mathcal{E}$ be a  Dirichlet form on $L^{2}(m)$, where $m$ is a $\sigma$-finite measure. The associated self-adjoint operator generates a Markovian $C_0$-semigroup on $L^2(m)$, which for $p \in [1,\infty)$ extends to a Markovian $C_0$-semigroup on   $L^p(m)$.

We first present a consequence of Fukushima's ergodic theorem \cite{Fuk82} for harmonic functions in the domain of the generator of the $L^p$-semigroup, which we also refer to as the $L^p$-generator of the Dirichlet form.

\begin{thm}[Basic $ L^{p} $-Liouville theorem]\label{t:Fuk} For  $p\in(1,\infty)  $ any harmonic function in the domain of the $L^{p}$-generator of an irreducible Dirichlet form is constant. In particular, in case there is a non-trivial harmonic function in the domain of the $L^p$-generator, then $m$ is a finite measure.	
\end{thm}

Next, we restrict ourselves to regular Dirichlet forms where one has an  intrinsic metric $ \rho $ in the sense of \cite{FLW} for which  all distance balls are precompact. The first result is an analogue of Yau's $L^{p}$-Liouville type theorem \cite{Yau76}. For some nonlocal operators on Euclidean space a related result is contained in  \cite{MU11}.  Speaking about \emph{(sub)harmonic} functions we refer here to weakly (sub)harmonic functions, see Definition~\ref{d:harmonic}.

\begin{thm}[Yau's $L^{p}$-Liouville theorem] \label{c:Yau} Let $ \mathcal{E} $ be an irreducible regular Dirichlet form without killing with an intrinsic metric for which all distance balls are precompact. Let $p\in(1,\infty)$ and let $ f \in  L^{p}(m)$  be  
	a non-negative subharmonic function. Then $f$ is constant if one of the following additional conditions is satisfied:
\begin{enumerate}[(a)]
 \item $1 < p \leq 2$.
 \item $m$ is finite.
 \item $p > 2$ and $f \in L^q(m)$ for some $q \in [2p-2,\infty]$. 
 \item $p > 2$, the intrinsic metric  has finite jump size and $f \in L^{2p-2}_{\rm loc}(m)$.
\end{enumerate}
\end{thm}

%

The next result is an analogue to Karp's $L^{p}$-Liouville theorem \cite{Karp82b}, which is found for strongly local regular Dirichlet forms in \cite{Sturm94} and for graphs in \cite{HuaJost13,HuaKeller13}.
Here we additionally need that the intrinsic metric has finite jump size, see Section~\ref{s:setup}. We denote the distance balls with radius $ r\ge0 $ about a fixed point $ o \in X $ with respect to the intrinsic metric $ \rho $ by $ B_{r}=\{x\in X\mid \rho(x,o)\leq r \} $.

\begin{thm}[Karp's $L^{p}$-Liouville theorem] \label{t:Karp}Let $ \mathcal{E} $ be an irreducible regular Dirichlet form without killing with an intrinsic metric for which all distance balls are precompact and the jump size is finite. Let $p \in (1,\infty)$ and let $q = \max\{p,2p-2\}$. Then every 
	non-negative subharmonic function $f\in L^{q}_{\rm loc}(m)$ satisfying
	\begin{align*}
		\int_{r_{0}}^{\infty} \frac{r}{\|f1_{B_{r}}\|_{p}^{p}}dr=\infty
	\end{align*}
	for all $r_0> 0  $ is constant.
\end{thm}

Other than in the above mentioned references in the case $p > 2$ we need some more integrability than $L^p$ for the analogue of Yau's theorem or $L^p_{\rm loc}$ for the analogue of Karp's theorem. The reason for this is a technical issue concerning the existence of certain integrals in our proof. In many concrete applications $L^\infty_{\rm loc}$-integrability (which is sufficient for our results) of  (sub)harmonic functions is known. This follows either from hypellipticity of the corresponding operators (which yields smoothness of harmonic functions) or, more generally, from local Hölder estimates for nonnegative subharmonic functions deduced by De Giorgi-Nash-Moser iteration. In this sense, the additional integrability assumptions we make can be seen as rather mild assumptions (at least when the jump size is finite).

In many special situations the positive and negative part of harmonic functions are positive subharmonic functions. This is for example the case for strongly local Dirichlet forms or Dirichlet forms on discrete sets. In these cases the results above direclty imply the corresponding results for harmonic functions.

We can use Karp's Liouville theorem to prove an integral criterion for recurrence.

\begin{thm}[Recurrence] \label{c:recurrence} Let $ \mathcal{E} $ be an irreducible regular Dirichlet form without killing with an intrinsic metric for which all distance balls are precompact and the jump size is finite. 
If
	\begin{align*}
		\int_{r_{0}}^{\infty}\frac{r}{m(B_{r})}dr =\infty
	\end{align*}
	for some $r_0 > 0$, then $\mathcal{E}$ is recurrent. 
\end{thm}


\section{Fukushima's ergodic theorem}\label{s:FET}

Let $m$ be a $\sigma$-finite measure on (a $\sigma$-algebra on) $X$. Let $ \mathcal{E} $ be a Dirichlet form on $ L^{2}(m) $ with nonnegative generator $ L $ and associated Markovian semigroup $ (T_{t}) = (e^{-tL})$. For $1 \leq p \leq \infty$, let $ (T^p_t)$ be the Markovian extension of the semigroup to $L^{p}(m) $. For $1 \leq p < \infty$, these are $C_0$-semigroups of contractions and we denote by  $ L_{p} $  the  corresponding generator, \cite{Davies}. For $p = \infty$, it is also a semigroup of contractions, which is only weak-*-continuous. These extension are compatible in the sense that for $p \neq q$ they agree on $L^p(m) \cap L^q(m)$. Hence, on the semigroup level we often omit the superscript $p$. 

In this section a function $ f \in  L^p(m) $ is called {\em $ L_{p}$-harmonic} if it belongs to the domain of $ L_{p} $ and $ L_{p}f = 0 $ holds.  For $ p = 2 $, we also speak about $ L $-harmonic functions instead of $ L_{2}$-harmonic functions.  Clearly, a function $f$  is $L$-harmonic if and only if $\mathcal{E} (f) = 0$.

As usual we denote the dual pairing between $ L^{p}(m) $ and $ L^{q}(m) $ with $1/q + 1/p = 1$ by $ (\cdot,\cdot) $.

Recall that a Dirichlet form $\cE$ is called {\em conservative} if $T_t 1 = 1$.   In 1982 Fukushima \cite{Fuk82} proved that if $ \mathcal{E} $ is  conservative (and coming from an $m$-symmetric Markov transition function), then for $p \in (1,\infty)$ and all $ f\in L^{p}(m) $
\begin{align*}
	\lim_{t \to \infty}T_{t}f =   g \quad m\text{-a.e.},
\end{align*}
 where $ g \in L^p(m)$ is a $ (T_{t})$-invariant function. In particular, if $ \mathcal{E} $ is additionally irreducible, then $ g $ is constant and equal to $ m(X)^{-1}\int_{X}fdm $ whenever $ m(X)<\infty $ and equal to $ 0 $ if $ m(X)=\infty $. Now, one can argue that $L_p$-harmonic functions are invariant under the semigroup and therefore all harmonic functions must be constant in the above setting.

Indeed, the setting of \cite{Fuk82} starts from an $ m $-symmetric, conservative Markov transition function on a $ \sigma $-finite measure space. The proof then relies on Rota's ergodic theorem \cite{DM80}. Here, we give an analytic proof of a related result under weaker assumptions. We only  need a (symmetric) Dirichlet form but  only  show strong convergence in $L^p(m)$ for $p \in (1,\infty)$. This is however enough to  establish Theorem~\ref{t:Fuk} along the same lines as discussed above.

%

\begin{lemma} \label{eigenfunction-constant} If $\cE$ is irreducible and $0$ is an eigenvalue of $L$, then any  corresponding eigenfunction must be constant.
\end{lemma}
\begin{proof} Let  $\varphi$ be  an eigenfunction to $0$.  As $\mathcal{E}$ is a Dirichlet form, we infer, for any normal contraction $C$,
	$$0 \leq \mathcal{E} (C \varphi) \leq \mathcal{E}(\varphi) =0$$
and, hence,
$$\mathcal{E} (C\varphi) =0.$$
  Thus, $C \varphi$ is an eigenfunction to $0$ for any normal contraction $C$. Now, as  $\cE$ is irreducible, we know that the eigenspace to $0$ is spanned by a unique function of fixed sign. The preceding  reasoning then shows that the span of this function  must be invariant under normal contractions. This is only possible if the function is constant.
\end{proof}

\begin{definition}[The ground state $\Phi$]
	If $\cE$ is irreducible, we define $\Phi$ to be zero if $0$ is not an eigenvalue of $L$ and to be  the unique  positive constant eigenfunction to $0$ with $\av{\Phi}_2 = 1$ otherwise.
\end{definition}

\begin{remark}
\begin{enumerate}[(a)]
 \item Let us emphasize  that $\Phi$   is a constant function in $L^2(m)$ (in all situations). Hence, $\Phi \neq 0$  can only occur if $m(X) < \infty$. In this case, constant functions are eigenfunctions to $0$ if and only if $\cE$ is conservative. Indeed, if $m(X) < \infty$ and $\cE$ is conservative, we have by definition $T_t1 = 1$ for all $t > 0$. This shows
 $$  \lim_{h \to 0+} h^{-1}(T_h 1 - 1) = 0$$
 in $L^2(m)$, so that $1 \in D(L)$ and $L1 = 0$.  Conversely, if $1 \in D(L)$ with $L1 = 0$, we have $1 \in D(\cE)$ and $\cE(1) = 0$. This implies conservativeness, see \cite[Theorem~1.6.6]{FOT}.

 In summary we obtain 
 $$\Phi = \begin{cases}
           \frac{1}{\sqrt{m(X)}} &\text{if } m(X) < \infty \text{ and } \cE \text{ is conservative},\\
           0 &\text{else}.
          \end{cases}
$$

 \item Since $\Phi \neq 0$ implies $m(X) <\infty$, the constant function $\Phi$ belongs to  $L^p(m)$ for any $p \in [1,\infty]$. Thus, in both cases $\Phi =0$ and $\Phi \neq 0$,  the map
	$$L^p(m) \to L^p(m),\quad f\mapsto (\Phi , f) \Phi$$
	is well defined and continuous. By our discussion above it is given by
	$$(\Phi , f) \Phi = \begin{cases}
           \frac{1}{m(X)} \int_Xf dm  &\text{if } m(X) < \infty \text{ and } \cE \text{ is conservative},\\
           0 &\text{else}.
          \end{cases}$$
\end{enumerate} 
\end{remark}

\begin{lemma}\label{eigenfunction-multiple} If $\cE$ is irreducible, then
	any $L$-harmonic function  is a multiple of $\Phi$.
\end{lemma}
\begin{proof}  This is clear (as any $L$-harmonic function either vanishes or is an eigenfunction to
	$0$).
\end{proof}

\begin{lemma} \label{convergence-q} If $\cE$ is irreducible, then for $p \in (1,\infty)$ and $f\in L^p (m)$
$$\lim_{t \to \infty}T_t f = ( \Phi, f )\Phi\qquad \text{ in } L^p(m).$$  
\end{lemma}
\begin{proof}  On the $L^2(m)$-level, the spectral theorem implies that  $T_t f$ converges to the projection of $f$ to the kernel of $L$, as $t \to \infty$. In our situation  this reads  
$$T_t f \stackrel{L^2}{\to} \langle \Phi, f\rangle \Phi,$$
 see e.g. \cite[Theorem~1.1]{KLVW}. For $f\in L^2(m)\cap L^\infty(m)\cap L^{1}(m) \subseteq L^p(m)$, we can estimate by Littlewood's   inequality  for $L^p$-spaces (i.e., Hölder inequality with a smart choice of parameters)
\begin{align*}
		\av{T_t f - (\Phi, f) \Phi}_p 
		&\leq   \av{T_t f - (\Phi, f) \Phi}^{1-\theta}_r \av{T_t f - (\Phi, f) \Phi}_2^\theta\\
		&\leq  C_r \av{f}_r^{1-\theta} \av{T_t f -  \langle \Phi, f\rangle \Phi}_2^{\theta}, 
\end{align*}
	with  $r = \infty$ and $\theta = 2/p$ if $p \geq 2$, and $r = 1$ and $\theta = 2(p-1)/p$, if $p \leq 2$. The second inequality follows from the fact that $(T_t)$ is a contraction on $L^1(m)$ and $L^\infty(m)$ and that the semigroups agree on $L^p(m) \cap L^2(m).$ 
	
This estimate and our discussion on the $L^2$-case show that the desired convergence holds on a dense subspace of $L^p(m)$. Since the semigroups are uniformly bounded, it extends to all of $L^p(m)$. 	 
	%
\end{proof}

 \begin{remark}
 Our discussion after the definition of $\Phi$ shows that this a version of Fukushima's ergodic theorem for semigroups associated with not necessarily conservative Dirichlet forms but with the weaker statement on $L^p$-convergence instead of $m$-a.e. convergence. 
 \end{remark}

\begin{lemma} \label{eigenfunction-semigroup}For $p \in [1,\infty)$, let $f \in L^p(m)$ be an $L_p$-harmonic function.  Then, $T_p f = f$ for any $t\geq 0$.
\end{lemma}\begin{proof} By abstract  theory for $C_0$-semigroups, for a given $g \in D(L_p)$,  the function $[0,\infty) \to L^p(m)$, $t \mapsto T_t g$ is the unique  solution in $C^1([0,\infty);L^p(m))$ of the Cauchy problem
	$$
	\begin{cases}
	  \dot u_t = L_p u_t \quad\text{ for } t >0\\
	u_0 = g
	\end{cases}.
	$$
   Now, obviously, $[0,\infty) \to L^p(m),\, t \mapsto  f$ is continuously differentiable and solves the problem with initial value $f$ so that $T_t f = f$ for $t > 0$.
\end{proof}

Theorem~\ref{t:Fuk} is a direct consequence of the following theorem.

\begin{thm} \label{main}Let $ \mathcal{E} $ be an irreducible Dirichlet form and let $f$ be an $L_p$-harmonic function for some $p \in (1,\infty)$. Then, $f$ is constant, and in fact
	$$ f = (\Phi, f)\Phi = \begin{cases}
           \frac{1}{m(X)} \int_Xf dm  &\text{if } m(X) < \infty \text{ and } \cE \text{ is conservative}\\
           0 &\text{else}
          \end{cases}.$$
\end{thm}
\begin{proof}  For $p \in (1,\infty)$, we infer from the previous two lemmas
	$$ f = T_t f \to (\Phi , f ) \Phi, \text{ as } t \to \infty. $$
	Combined with our computation of  $(\Phi , f ) \Phi$ this yields the result.
\end{proof}


\section{Regular Dirichlet forms and harmonic functions}\label{s:setup}


\subsection{Basic notions and intrinsic metrics}

We use the notation of the previous section. Moreover, from now on we additionally assume that $X$ is a locally compact separable metric space,  $m$ is a Radon measure of full support and $\mathcal{E}$ is a regular Dirichlet form on $L^2(m)$, see \cite{FOT}. 

By $C(X)$ we denote the space of continuous functions on $X$ and by $C_c(X)$ the space of continuous functions of compact support.




As for continuous functions, we write $D(\cE)_c$ for the functions in $D(\cE)$ with compact support. We denote by $D(\cE)_{\rm loc}$ the space of functions locally in $D(\cE)$, that is the set of functions $f$ such that for every open precompact set $G$ there is a function in $D(\cE)$ which coincides with $f$ on $G$.


By the Beurling-Deny formula \cite[Theorems~3.2.1 and 5.2.1]{FOT} we have for $f,g\in D(\cE)$
$$\cE(f)=\int_X d\Gm^{(c)}(f)+\int_{X\times X \setminus d} (\tilde{f}(x)-\tilde{f}(y))^2 dJ(x,y)+\int_{X}\tilde{f}(x)^2 k(dx),$$
where $\Gm^{(c)}$  is a   measure valued  strongly local quadratic form on $D(\cE)$ (i.e., $\Gm^{(c)}(f,g)=0$ if $f$ is constant on a neighborhood of the support of $g$), $J$ is a Radon measure on $X\times X \setminus d=\{(x,y)\in X\times X\mid x\neq y\}$ and $k$ is a Radon measure on $X$. Moreover, $\tilde f$ is a  quasi-continuous representative of $f$. Such a representative exists for $f\in D(\cE)$ by \cite[Theorem~2.1.7]{FOT} and the argument given there extends directly to $ f\in D(\cE)_{\rm loc} $, see e.g.  \cite[Proposition 3.1]{FLW}. To simplify notation, below we will   always choose a quasi-continuous representative and just write $f$ instead of $\tilde f$. Since we exclusively deal with continuous functions or functions from  $D(\cE)_{\rm loc}$, this is always possible. \medskip

{\bf Assumption:} From now on we assume that $\cE$ has no killing, i.e., $k = 0$.  
\medskip

Using the measure $J$ from the decomposition we obtain a finite Radon measure valued quadratic form $\Gm^{(j)}$ on $D(\cE)$, which is uniquely determined by
\begin{align*}
	\int_{K}d\Gm^{(j)}(f) =\int_{K \times X \setminus d} (f(x) - f(y))^2 dJ(x,y)
\end{align*}
%
%
for all compact $K\subseteq X$ and $f\in D(\cE)$. 

Next, we discuss how to extend $\Gm^{(c)}$ and $\Gm^{(j)}$ to larger classes of functions. 
%
%
%
%
%
 We denote by $ D(\cE)_{\rm loc}^{*}$ the space of functions in $f\in D(\cE)_{\rm loc}$ such that for every compact $K\subseteq X$
 \begin{align*}
   \int_{K\times X \setminus d}(f(x)-f(y))^{2}dJ(x,y)<\infty.
\end{align*}
This space was introduced in \cite{FLW}. For later purposes, we note that any $f \in D(\cE)_{\rm loc}^{*}$ with compact support belongs to $D(\cE)_{\rm c}$, see   \cite[Theorem~3.5]{FLW}.

The bilinear forms on $D(\cE) \times D(\cE)$ induced from $\Gm^{(c)}$ and $\Gm^{(j)}$ by polarization can be (partially) extended to bilinear forms $D(\cE)_{\rm loc}^* \times D(\cE)_c$ with values in the space of finite signed Radon measures.  More precisely, for $f \in D(\cE)_{\rm loc}^*$ and $\varphi \in D(\cE)_c$ these extensions are characterized by
$$\int_{K}d\Gm^{(c)}(f,\varphi) = \int_{K}d\Gm^{(c)}(f_K,\varphi), $$
and
$$\int_{K}d\Gm^{(j)}(f,\varphi) =\int_{K \times X \setminus d} (f(x) - f(y))(\varphi(x) - \varphi(y)) dJ(x,y),$$
%
for all compact $K \subseteq X$. Here, $f_K \in D(\cE)$ is a function with $f = f_K$ on a relatively compact open neighborhood of $K$. Due to the strong locality $\Gm^{(c)}(f,\varphi)$ is well definied (i.e., independent of the choice of $f_K$, see the remarks after the proof of \cite[Theorem 3.2.2]{FOT}) and the definition of  $D(\cE)_{\rm loc}^*$ ensures that $\Gm^{(j)}(f,\varphi)$ is well defined (i.e., it yields a finite signed Radon measure, see \cite[Theorem~3.4]{FLW}). Similarly, the quadratic forms $\Gm^{(c)}$ and $\Gm^{(j)}$ can be extended to quadratic forms on $D(\cE)_{\rm loc}^*$ taking values in the set of nonnegative (not necessarily finite) Radon measures, see \cite[Proposition~3.3]{FLW}. We let $\Gamma = \Gm^{(c)} + \Gm^{(j)}$ denote both the quadratic form on $D(\cE)_{\rm loc}^*$ and the induced bilinear form on $D(\cE)_{\rm loc}^* \times D(\cE)_c$. 

Since $\cE$ is a Dirichlet form, for any normal contraction $C \colon \R \to \R$ and $f \in \dEl^*$, we have $C \circ f \in \dEl^*$ and 
$$\int_X d\Gamma(C \circ f) \leq \int_X d\Gamma(f).$$
Since the killing $k$ vanishes, this inequality extends to $1$-Lipschitz functions.  Indeed, if $C \colon \R \to \R$ is  $1$-Lipschitz, then $C \circ f = \tilde C (f) + C(0)$ with the normal contraction $\tilde C = C - C(0)$. Since the constant function $1$ belongs to $\dEl^*$, see \cite[Proposition~3.1]{FLW} and satisfies $\Gamma(1) = 0$, this yields the claim.

The strongly local part satisfies a Leibniz rule. If $f,g \in \dEl^*$ such that $fg \in \dEl^*$, then 
$$d\Gamma^{(c)}(fg,\varphi) = f d\Gamma^{(c)}(g,\varphi) + g d\Gamma^{(c)}(f,\varphi)$$
for all $\varphi \in D(\cE)_c$, see the proof of \cite[Theorem~3.9]{FLW}. Moreover, it satisfies a chain rule. If $C \in C^1(\R)$ with bounded derivative, then for $f \in \dEl^*$, we have $C \circ f \in \dEl^*$ and 
$$\Gm^{(c)}(C(f),\varphi) = C'(f) \Gm^{(c)}\Gamma(f,\varphi)$$
for all $\varphi \in D(\cE)_c$, see \cite[Theorem~3.2.2]{FOT}.


Let $\rho \colon X\times X\to[0,\infty)$ be a pseudo metric, that is a symmetric map with zero diagonal satisfying the triangle inequality. For measurable sets $A\subseteq X$, we write $\rho_{A}=\inf_{x\in A}\rho(x,\cdot)$. The following definition is a slight modification of the one given in  \cite{FLW}.

\begin{definition}[Intrinsic metric]\label{d:intrinsic} A pseudo metric $\rho$ on $X$ is called   {\em intrinsic} (for the regular Dirichlet form $\cE$) if it is continuous with respect to the topology on $X$ and there are Radon measures $m^{(c)}$, $m^{(j)}$ with $m^{(c)}+m^{(j)}\leq m$  such that for every measurable $A \subseteq X$ the following holds: $\rho_{A}\in  D(\cE)_{\rm loc}^{*}$  and   
	\begin{align*}
		\Gm^{(c)}(\rho_{A})\leq m^{(c)}\quad\mbox{and}\quad \Gm^{(j)}(\rho_{A})\leq m^{(j)}.
	\end{align*}
\end{definition}

\begin{remark}
 In \cite{FLW}, $\rho$ is allowed to take the value infinity. In this sense our definition is a bit more restrictive. 
\end{remark}

\begin{definition}\label{d:C} The \emph{jump size} $s$ of a continuous pseudo metric $\rho$ with respect to the Dirichlet form $\cE$ is given by
	\begin{align*}
		s:=\inf\{t\ge0\mid J(\{(x,y)\in X\times X \setminus d\mid \rho(x,y)>t\})=0\} \in [0,\infty].
	\end{align*}
\end{definition}

If $\rho$ is a fixed pseudo-metric, for  $U\subseteq X$  and $r \in\mathbb{R}$ we write  
$$B_{r}(U)=\{x\in X\mid\rho(x,y)\leq r \mbox{ for some }y\in U\}.$$

The relevance of intrinsic metrics comes from the fact that Lipschitz functions with respect to intrinsic metrics induce good cut-off functions. This is discussed next. 

Fix $o \in X$ and let $0\leq r<R$. Below we will make use of  
\begin{align*}
\eta = \eta_{r,R} \colon X \to \R,\quad \eta(x)=1\wedge \left(\frac{R-\rho(x,o)}{R-r}\right)_{+}.
\end{align*}
It is $(R-r)^{-1}$-Lipschitz and satisfies $\eta = 1$ on $B_r$ and $\eta =  0$ on $X\setminus B_{R}$.  
\begin{lemma}\label{lemma:lipschitz-cutoff}
 Let $\rho$ be an intrinsic metric for $\cE$ with jump size $s$. Then $\eta \in D(\cE)_{\rm loc}^{*}$ and it satisfies
  $$\Gamma^{(c)}(\eta) \leq \frac{1}{(R-r)^2} 1_{B_R \setminus B_r} m^{(c)} \text{ and } \Gamma^{(j)}(\eta) \leq \frac{1}{(R-r)^2}1_{B_{R + s} \setminus B_{r-s}} m^{(j)}.$$
  If, moreover, balls with respect to $\rho$ are precompact, then $\eta \in D(\cE)_c$. 
\end{lemma}
\begin{proof}
  The first statement and the estimates are the content of \cite[Theorem~4.9 and Proposition~8.5]{FLW}. The last statement follows from the fact that functions in $D(\cE)_{\rm loc}^{*}$ with compact support belong to $D(\cE)_c$.
\end{proof}
If $\rho$-balls are precompact, the previous lemma and the definitions of $\Gm^{(c)}$ and $\Gm^{(j)}$ yield for  $ f \in \dEl^* $
$$\int_{X}d\Gm^{(c)}(f,\eta)=\int_{B_{R}\setminus B_{r}}d\Gm^{(c)}(f,\eta),$$
and 
$$\int_{X \times X \setminus d} (f(x)-f(y)) (\eta(x) - \eta(y))dJ(x,y)  = \int_{U_{r,R}} (f(x)-f(y)) (\eta(x) - \eta(y))dJ(x,y) $$
with $U_{r,R} = (B_{R+s} \setminus B_{r-s}) \times (B_{R+s} \setminus B_{r-s}) \setminus d$, whenever the integral makes sense. 

%
%
%
%
%
%
%
%

\subsection{Harmonic functions}

In this subsection we introduce (sub)harmonic functions and discuss their basic properties. 

\begin{definition}\label{d:harmonic} A function $f:X\to\R$ is called \emph{harmonic}  (\emph{subharmonic}) if $f\in D(\cE)_{\rm loc}^*$ and
	\begin{align*}
		\int_{X}d\Gm(f,\ph)=0\quad \mbox{($\le0$),}
	\end{align*}
	for all $0\leq \ph\in  D(\cE)_c$.
\end{definition}

\begin{remark}
 Since domains Dirichlet forms are lattices, a harmonic function $f$ statisfies $\int_{X}d\Gm(f,\ph)=0$ for all $\varphi \in D(\cE)_c$.
\end{remark}


Below we will need the following extension of Lemma~\ref{eigenfunction-constant} to $\dEl^*$ functions. It can also be viewed as a Liouville-type theorem.

\begin{lemma}\label{lemma:gamma 0 yields constant}
 Let $\cE$ be irreducible. Then any $f \in \dEl^*$ with 
 $\Gamma(f) = 0$
 is constant. 
\end{lemma}
\begin{proof}
 We show that for each $\alpha > 0$ the set $A_\alpha =  \{f > \alpha\}$ is $\cE$-invariant. Since $ \cE $ is irreducible, this is only possible if $f$ is constant. To this end, we show $ \cE(1_{A_\alpha} \varphi) \leq \cE(\varphi) $ which implies the $ \cE $-invariance, see e.g. \cite[Lemma~2.32]{Schmidt17}.
 
 We have 
 $$1_{A_\alpha} = \lim_{n\to \infty} (n (f-\alpha)_+) \wedge 1 $$
 pointwise $m$-a.e. Let $g_n = (n (f-\alpha)_+) \wedge 1$. Up to a constant $g_n$ is a contraction of $f$ and so we have $$ \Gamma(g_n) \leq  n^2 \Gamma(f) = 0 .$$
 
 Now let $\varphi \in D(\cE) \cap C_c(X)$ be given.  Then $\varphi g_n \in D(\cE)$ (here we use that  $\varphi g_n$ has compact support and belongs to $\dEl^* \cap L^\infty(m)$ by \cite[Proposition~3.10]{FLW}). The Leibniz rule for $\Gm^{(c)}$, the ``discrete Leibniz rule'' (i.e., $f(x)g(x) - f(y)g(y) = f(x)(g(x)-g(y)) + g(y)(f(x) - f(y))$ for functions $f,g$ on $X$ and $x,y \in X$) and the Cauchy-Schwarz inequality yield
 \begin{align*}
 &\cE(\varphi g_n)  = \int_X \varphi d\Gamma(g_n,g_n\varphi) + \int_X g_n d\Gamma(\varphi,g_n\varphi)\\
 &\leq \left(\int_X d\Gamma(g_n)\right)^{1/2} \left(\int_X |\varphi|^2  d\Gamma(\varphi g_n)\right)^{1/2} + \left(\int_X |g_n|^2 d\Gamma(\varphi)\right)^{1/2} \left(\int_X   d\Gamma(\varphi g_n)\right)^{1/2}\\
 &\leq \cE(\varphi)^{1/2} \cE(\varphi g_n)^{1/2}.
 \end{align*}
 This shows $\cE(\varphi g_n) \leq \cE(\varphi)$. Since by dominated convergence $1_{A_\alpha} \varphi = \lim_{n\to \infty}  \varphi g_n$ in $L^2(m)$, the lower semicontinuity of $\cE$ and this inequality yield $1_{A_\alpha} \varphi \in D(\cE)$ and 
$$\cE(1_{A_\alpha} \varphi) \leq \cE(\varphi).$$
Since $D(\cE) \cap C_c(X)$ is dense in $D(\cE)$, this inequality extends to all $\varphi \in D(\cE)$ (use lower semicontinuity) and we obtain the $\cE$-invariance of $A_\alpha$. 
\end{proof}
%


\section{A  Caccioppoli inequality}

The following Caccioppoli-type inequality is the key to Yau's $L^{p}$-Liouville theorem. A somewhat more subtle estimate is needed for Karp's result. Throughout $\cE$ is a regular Dirichlet form without killing and $\rho$ is an intrinsic metric for $\cE$.

\begin{thm}[Caccioppoli-type inequality] \label{t:Caccioppoli} Assume the distance balls of $\rho$ are precompact and $p\in (1,\infty)$. Then, there is $C>0$ such that for  every non-negative subharmonic function $f$  and all $0<r<R$ such that $f1_{B_{R + s}\setminus B_{r - s}} \in L^q(m)$ for some $q \in [2p-2,\infty]$ we have
	\begin{align*}
		\int_{B_r}f^{p-2}d\Gm^{(c)}(f)+\int_{X\times B_r \setminus d}(f(x)\vee f(y))^{p-2}
		&(f(x)-f(y))^{2}dJ(x,y)\\
		&\leq \frac{C}{(R-r)^{2}}\av{f1_{B_{R+s}\setminus
			B_{r-s}}}_{p}^{p},
	\end{align*}
	where $s$ is the jump size of $\rho$. 
\end{thm}

\begin{remark}
 \begin{enumerate}[(a)]
  \item  Note that $p \leq 2$ if and only if $2p - 2 \leq p$. Hence, in this case,  $f1_{B_{R + s}\setminus B_{r - s}} \in L^q(m)$ for some $q \geq 2p-2$ is already satisfied if the right side of the Caccioppoli-type inequality is finite.
  
  \item If the jump-size of $\rho$ is finite, the set $B_{R+s}$ is precompact and hence has finite measure.  In this case, for $p > 2$ the assumption on the integrability on $f$ is satisfied for all $0 \leq r < R$ if and only if $f \in L^{2p-2}_{\rm loc}(m)$.
 \end{enumerate}
 
\end{remark}

\begin{remark}For $p\ge2$ the inequality above yields (with a larger constant $C$ and under suitable conditions on $f$)
	\begin{align*}
		\int_{X}f^{p-2}d\Gm(f)&\leq \frac{C}{(R-r)^{2}}\av{f1_{B_{R+s}\setminus
			B_{r-s}}}_{p}^{p}.
	\end{align*}
\end{remark}

Let $\mathrm{Lip}_{c}(X)$ be the space of Lipshitz continuous functions with compact support with respect to the intrinsic metric $\rho$.


\begin{lemma}[The key estimate]\label{l:keyestimate} 
	Let $p\in (1,\infty)$ and let $n \in \N$. For every non-negative subharmonic function $f$   and $\ph\in \mathrm{Lip}_{c}(X)$, we have
	\begin{align*}
		&\int_{A_n}f^{p-2}\ph^{2} d\Gm^{(c)}(f)+\int_{X\times X \setminus d}(f_n(x)\vee f_n(y))^{p-2}\ph(y)^{2}(f_n(x)-f_n(y))^{2}dJ(x,y)\\
		&\leq -C \int_{X} f_n^{p-1}\ph d\Gm^{(c)}(f,\ph) \\
		&\qquad - C\int_{X\times X \setminus d} f_n^{p-1}(x)\ph(y)(f(x)-f(y))(\ph(x)-\ph(y))dJ(x,y) < \infty,
	\end{align*}
	where $f_n  =  f\wedge n$, $A_n = \{  f < n\}$ and $C=2/((p-1)\wedge 1)$.
\end{lemma}
\begin{proof}
	For $j \in \N$, we let $g_j = f_n \vee j^{-1}$. We show the inequality with $f_n$ replaced by $g_j$ and $A_n$ replaced by $\{j^{-1} < f < n\}$. Then  the statement follows from Fatou's lemma and Lebegue's dominated convergence theorem after letting $j \to \infty$. Note that the integral on the right hand side of the equation exists due to the boundedness of $f_n$ and the definition of $\dEl^*$.

	Since $g_j \in \dEl^*$ is bounded from below by $1/j$ and from above by $n$, we have $g_j^{p-1} \in \dEl^*$ (here we use that $[j^{-1},n] \to \R, t \mapsto t^{p-1}$ is Lipschitz). Using the boundedness of $g_j^{p-1}$, for $\ph\in \mathrm{Lip}_{c}(X) \subseteq D(\cE)_c \cap L^\infty(m)  $ we   obtain   $ \ph^{2}g_{j}^{p-1} \in D(\cE)_c$ as $ D(\cE)_c \cap L^\infty(m) $ is an ideal in the algebra $ D(\cE) \cap L^\infty(m) $. The  subharmonicity and non-negativity  of $f$ yields
	\begin{align*}
		0\geq \int_{X}d\Gm^{(c)}(f,\ph^{2}g^{p-1}_{j}) + \int_{X}d\Gm^{(j)}(f,\ph^{2}g^{p-1}_{j}).
	\end{align*}
To shorten notation, we write $ \nabla_{xy}g=g(x)-g(y) $ for $ x,y\in X $ and functions $ g $. 
	We apply the ``discrete Leibniz rules'' $\nabla_{xy}(fg) = g(x) \nabla_{xy}f +f(y) \nabla_{xy} g$ and $\nabla_{xy}f^2 = 2f(y) \nabla_{xy}f + (\nabla_{xy}f)^2$ to the jump term
	\begin{align*}
		\lefteqn{ \int_{X}d\Gm^{(j)}(f,\ph^{2}g^{p-1}_{j})}\\
		=&\int_{X\times X \setminus d}\nabla_{xy} f\cdot \big(\ph^{2} (y) \nabla_{xy} g^{p-1}_{j}+g^{p-1}_{j}(x)\nabla_{xy}\ph^{2}\big)dJ(x,y)\\
		=&\int_{X\times X\setminus  d}\nabla_{xy} f
		\cdot\big(\ph^{2} (y) \nabla_{xy} g^{p-1}_{j} +2g^{p-1}_{j}(x)\ph(y)\nabla_{xy}\ph +g^{p-1}_{j}(x)(\nabla_{xy}\ph)^{2}\big)dJ(x,y).
	\end{align*}
We start by estimating the first term, where we apply \cite[Lemma~2.9]{HuaKeller13}, which states that if $ \nabla_{xy} g\ge0 $ for a non-negative function $ g $, then
\begin{align*}
	\nabla_{xy} g^{p-1} \ge   C(g(x)\vee g(y))^{p-2}\nabla_{xy}g
\end{align*}
for the constant $ C=(p-1)\wedge 1 $. Since  $ \nabla_{xy}f\ge 0 $ implies $ \nabla_{xy}g^{p-1}_{j}\ge0 $, 
$ \nabla_{xy}f\nabla_{xy} g_{j}^{p-1}=\nabla_{yx}f\nabla_{yx}g_{j}^{p-1}\ge0 $ and $ \nabla_{xy}f\nabla_{xy} g_{j}=\nabla_{yx}f\nabla_{yx}g_{j} \geq |\nabla_{xy} g_{j}|^2$, this gives
	\begin{align*}
		\ph^{2} (y)\nabla_{xy} f\cdot\nabla_{xy} g^{p-1}_{j} \ge  C \ph^{2} (y)(g_{j}(x)\vee g_{j}(y))^{p-2} |\nabla_{xy} g_{j}|^2
	\end{align*}
for the first term in the estimate above. We leave the second term as it is for now. For third term  we obtain  with having $  \nabla_{xy}f\nabla_{xy} g_{j}\ge0 $ in mind
	\begin{align*}
		0&\leq\int_{X\times X \setminus d} \nabla_{xy} g^{p-1}_{j}\nabla_{xy} f(\nabla_{xy}\ph)^{2}dJ(x,y)\\
		&=2\int_{X\times X \setminus d} g^{p-1}_{j}(x)\nabla_{xy} f(\nabla_{xy}\ph)^{2}dJ(x,y),
\end{align*}
	where the equality holds as the term on the right hand side converges absolutely by Hölder's inequality. Hence, we obtain
	\begin{align*}
		\int_{X}d\Gm^{(j)}(f,\ph^{2}g^{p-1}_{j})\geq& C\int_{X\times X \setminus d}\ph^{2} (y)(g_{j}(x)\vee g_{j}(y))^{p-2}  |\nabla_{xy} g_{j}|^2 dJ(x,y) \\ &\quad+2\int_{X\times X \setminus d}g^{p-1}_{j}(x)\ph(y)\nabla_{xy} f\nabla_{xy}\ph dJ(x,y).
	\end{align*}
	Applying the Leibniz rule and the chain rule to the strongly local term we get immediately
	\begin{align*}
		\int_{X}d\Gm^{(c)}(f,\ph^{2}g^{p-1}_{j})&= (p-1)\int_{X}\ph^{2}g^{p-2}_{j}d\Gm^{(c)}(f,g_{j}) +2\int_{X}g^{p-1}_{j}\ph d\Gm^{(c)}(f,\ph)\\
		&=(p-1)\int_{\{m^{-1} < f < n\}}\ph^{2}f^{p-2} d\Gm^{(c)}(f) +2\int_{X}g^{p-1}_{j}\ph d\Gm^{(c)}(f,\ph).
	\end{align*}
 	For the last equality we used the truncation property of $\Gamma^{(c)}$ discussed in \cite[Appendix~4.1]{Sturm94}. Putting the two estimates for the terms $\int_{X}d\Gm^{(j)}(f,\ph^{2}g^{p-1}_{j})$ and  $\int_{X}d\Gm^{(c)}(f,\ph^{2}g^{p-1}_{j})$ into the inequality at the beginning of the proof yields the statement.
\end{proof}

\begin{proof}[Proof of the Caccioppoli-type inequality, Theorem~\ref{t:Caccioppoli}] Let $\eta = \eta_{r,R}$ be the cut-off functions discussed in Lemma~\ref{lemma:lipschitz-cutoff}. We use $\varphi = \eta = \eta_{r,R}$   in our key estimate and deal with the terms on the right-hand side separately. Without loss of generality we can assume $f 1_{B_{R+s} \setminus B_{r-s}} \in L^p(m)$ for otherwise there is nothing to show. We further assume $q < \infty$ as it will become clear in the course of the proof that the case $q = \infty$ is simpler and follows along the same lines. 

We use the inequality $|Q(u,v)| \leq \varepsilon Q(u,u)^2 + \frac{1}{4\varepsilon} Q(v,v)^2$, which holds for nonnegative bilinear forms $Q$ and $\varepsilon > 0$, and $ \eta\le 1 $  to estimate
\begin{align*}
		C \left| \int_{A_n} f_n^{p-1}\eta d\Gm^{(c)}(f,\eta)\right| &\leq \frac{1}{2} \int_{A_n} f_n^{p-2}\eta^2 d\Gamma^{(c)}(f) + C' \int_{A_n} f^p d\Gamma^{(c)}(\eta) \\
		&\leq \frac{1}{2} \int_{A_n} f_n^{p-2}\eta^2 d\Gamma^{(c)}(f) + \frac{C'}{(R-r)^2} \av{f^p 1_{B_R \setminus B_r}}_p^p,
	\end{align*}
where $ f_{n}=f\wedge n $ and $ A_{n}=\{f<n\} $.
%
%
%
%
Moreover, $f_n^{p-1} = n^{p-1}$ on $X\setminus A_n$, Cauchy-Schwarz inequality, the cut-off properties of $\eta$ and Chebyshev's inequality  yield 
\begin{multline*}
		\left| \int_{X \setminus A_n} f_n^{p-1}\eta d\Gm^{(c)}(f,\eta)\right| \leq  n^{p-1}\left(\int_{X \setminus A_n}  d\Gm^{(c)}(\eta)\right)^{1/2}\left(\int_{X \setminus A_n} \eta^2 d\Gm^{(c)}(f)\right)^{1/2}\\
		\leq \frac{n^{p-1}}{(R-r)}m( \{f \geq n\} \cap B_R\setminus B_r )^{1/2}\left(\int_{X \setminus A_n} \eta^2 d\Gm^{(c)}(f)\right)^{1/2} \\
		\leq  \frac{n^{p-1 - q/2}}{(R-r)} \av{f1_{B_R \setminus B_r}}^{q/2} \left(\int_{X \setminus A_n} \eta^2 d\Gm^{(c)}(f)\right)^{1/2}. 
	\end{multline*}
Since $\infty > q \geq 2p-2$ and $f1_{B_R \setminus B_r} \in L^q(m)$, the right side of this inequality converges to $0$, as $n \to \infty$.  A similar reasoning yields 
\begin{align*}
 &C\left|\int_{A_n\times X \setminus d} f_n^{p-1}(x)\eta(y)(f(x)-f(y))(\eta(x)-\eta(y))dJ(x,y)\right|\\
 &\leq C\int_{A_n\times X \setminus d} (f_n(x) \vee f_n(y))^{p-1}\eta(y)|f(x)-f(y)||\eta(x)-\eta(y)|dJ(x,y)\\
 &\leq \frac{1}{2} \int_{A_n\times X \setminus d} (f_n(x) \vee f_n(y))^{p-2}\eta^2(y)(f(x)-f(y))^2 dJ(x,y)\\
 &\quad + C' \int_{A_n\times X \setminus d} (f_n(x) \vee f_n(y))^{p}(\eta(x)-\eta(y))^2 dJ(x,y)\\
  &\leq \frac{1}{2} \int_{A_n\times X \setminus d} (f_n(x) \vee f_n(y))^{p-2}\eta^2(y)(f(x)-f(y))^2 dJ(x,y)\\
 &\quad + \frac{C''}{(R-r)^2} \av{f^p1_{B_{R + s} \setminus B_{r-s}}}_p^p. 
\end{align*}
For the last inequality we used the properties of $\eta$ and $(a \wedge b)^p \leq a^p + b^p$ as well as the symmetry of $J$. Moreover, as for the strongly local part we obtain using the cut-off properties of $\eta$  
\begin{align*}
 &\int_{(X \setminus A_n)\times X\setminus d} f_n(x)^{p-1}\eta(y)|f(x)-f(y)||\eta(x)-\eta(y)|dJ(x,y)\\
 &\leq  \frac{n^{p-1 - q/2}}{R-r} \av{f1_{B_{R+s} \setminus B_{r-s}}}^{q/2} \left(\int_{X \setminus A_n} \eta^2 d\Gm^{(j)}(f)\right)^{1/2},
\end{align*}
with the right side converging to $0$, as $n \to \infty$. 

Plugging these estimates into our key estimate in Lemma~\ref{l:keyestimate}, using $f = f_n$ on $A_n$ and $\eta = 1$ on $B_r$, yields a constant $C > 0$ such that for all $n \in \N$  
\begin{align*}
 &\int_{A_n \cap B_r}f^{p-2} d\Gm^{(c)}(f)+\int_{A_n\times (A_n \cap B_r) \setminus d}(f(x)\vee f(y))^{p-2} (f(x)-f(y))^{2}dJ(x,y)\\
		&\quad \leq \frac{C}{(R-r)^2} \av{f1_{B_{R+s} \setminus B_{r-s}}}_p^p + E_n,
\end{align*}
with $E_n \to 0$, as $n \to \infty$. With this at hand the statement follows after letting $n \to \infty$. 

In the case $q = \infty$ we can use the same estimates as above without invoking the limit $n \to \infty$, as in this case $A_n = X$ for some $n \in \N$ and $E_n = 0$ for some $n \in \N$.
\end{proof}

 For  proving Yau's theorem the Caccioppoli-type inequality is sufficient. For Karp's theorem we need one further estimate, which follows from our key estimate and the Caccioppoli-type inequality. Here,  $\eta = \eta_{r,R}$ denotes the cut-off function discussed in Lemma~\ref{lemma:lipschitz-cutoff}.

\begin{lemma}\label{lemma:additional inequality}
Assume the distance balls of $\rho$ are precompact and the jump-size $s$ is finite.  For  $p\in (1,\infty)$ there is $C>0$ such that for  every non-negative subharmonic function $f \in L^{q}_{\rm loc}(m)$ with $q = \max\{p,2p-2\}$ and all $0<r<R$   we have
	\begin{multline*}
		\left(\int_{X} f^{p-2} \eta^2 d\Gm^{(c)}(f)+\int_{X \times X \setminus d}(f(x)\vee f(y))^{p-2} \eta^2(y)
		(f(x)-f(y))^{2} dJ(x,y) \right)^2\\
		\leq \frac{C}{(R-r)^{2}}\av{f1_{B_{R+s}\setminus B_{r-s}}}_{p}^{p} \Biggl(\int_{B_R \setminus B_r}  f^{p-2} \eta^2 d\Gm^{(c)}(f) \\
		\qquad+ \int_{U_{R+s} \setminus U_{r-s} }(f(x)\vee f(y))^{p-2} \eta^2(y)
		(f(x)-f(y))^{2} dJ(x,y) \Biggr),
	\end{multline*}
	with $\eta = \eta_{r,R}$ and $U_r = B_r \times B_r \setminus d$.
\end{lemma}
\begin{proof}
 We apply our key estimate  Lemma~\ref{l:keyestimate} to $f$ with the cut-off function $\eta = \eta_{r,R}$. Then, the Cauchy-Schwarz inequality and the cut-off properties of $\eta$ yield 
 \begin{multline*}
		\left(\int_{A_n}f^{p-2}\eta^{2} d\Gm^{(c)}(f)+\int_{X\times X \setminus d}(f_n(x)\vee f_n(y))^{p-2}\eta(y)^{2}(f_n(x)-f_n(y))^{2}dJ(x,y)\right)^2\\
		 \leq \frac{C}{(R-r)^{2}}\av{f_n 1_{B_{R+s}\setminus B_{r-s}}}_{p}^{p} \Biggl(\int_{B_R \setminus B_r}  f_n^{p-2} \eta^2 d\Gm^{(c)}(f) \\
	\qquad+ \int_{U_{R+s} \setminus U_{r-s} }(f_n(x)\vee f_n(y))^{p-2} \eta^2(y)
		(f(x)-f(y))^{2} dJ(x,y) \Biggr).
	\end{multline*}
	 By our Caccioppoli-type inequality Theorem~\ref{t:Caccioppoli} all of the   integrals on the right side exist when $f_n$ is replaced by $f$ (here we use that $B_{R+s}$ is precompact, $0\leq \eta \leq 1$ with $\eta = 0$ on $X \setminus B_R$ and the integrability assumption on $f$). Hence, we can take the limit $n \to \infty$ to obtain the statement. 
\end{proof}

%

\section{Proof of Yau's and Karp's theorem and recurrence}
In this section we use the results from the previous sections to prove  Yau's and Karp's theorem, Theorem~\ref{c:Yau} and~\ref{t:Karp}. Later in this section we prove the growth test for recurrence, Theorem~\ref{c:recurrence}.

\subsection{Proof of Yau's and Karp's theorem}
The proofs of our Liouville theorems are based on the following observation.

\begin{lemma}\label{lemma:irreducible constant}
 Let $\cE$ be irreducible and let $f \in \dEl^*$ be non-negative. If for some $p \in (1,\infty)$
 $$\int_X f^{p-2} d\Gm^{(c)}(f) + \int_{X \times X \setminus d} (f(x) \vee f(y))^{p-2} (f(x) - f(y))^2 dJ(x,y) = 0,$$
 then $f$ is constant. 
\end{lemma}

\begin{proof}
By the truncation property of the strongly local measure $\Gm^{(c)}$ and $f\geq 0$ we have 
$$\Gm^{(c)}(f) = \Gm^{(c)}(f_+)= 1_{\{f > 0\}}\Gm^{(c)}(f), $$
see \cite[Appendix~4.1]{Sturm94}. Hence,  $\int_X f^{p-2} d\Gm^{(c)}(f) = 0$ implies $\Gm^{(c)}(f) = 0$. Similarly, 
$\int_{X \times X \setminus d} (f(x) \vee f(y))^{p-2} (f(x) - f(y))^2 dJ(x,y) = 0$
implies $\int_{X \times X \setminus d}  (f(x) - f(y))^2 dJ(x,y) = 0$. These two observations and Lemma~\ref{lemma:gamma 0 yields constant} show the claim. 
\end{proof}

\begin{proof}[Proof of Yau's $L^{p}$-Liouville theorem, Theorem~\ref{c:Yau}]  The assumptions we made guarantee that we can apply the Caccioppoli-type  inequality Theorem~\ref{t:Caccioppoli} for all $0 < r < R$. In this inequality letting first $R \to \infty$ and then $r \to\infty$ yields
$$\int_X f^{p-2} d\Gm^{(c)}(f) + \int_{X \times X \setminus d} (f(x) \vee f(y))^{p-2} (f(x) - f(y))^2 dJ(x,y) = 0.$$
Hence, the previous lemma implies that $f$ is constant. 
\end{proof}

\begin{proof}[Proof of Karp's theorem, Theorem~\ref{t:Karp}] With Lemma~\ref{lemma:additional inequality} at hand we basically follow   \cite{Sturm94}. Suppose $f \neq 0$. Let $R \geq  4s$ be such that $f 1_{B_{R}} \neq 0$. For $n \in \N$, we define $R_n = 2^n R$, $\eta_n = \eta_{R_{n-1}+s,R_n-s}$, $v_n = \av{f 1_{B_{R_n} \setminus B_{R_{n-1}}}}_p^p$ and 
$$Q_n = \int_{B_{R_n}} f^{p-2} \eta_{n}^2 d\Gm^{(c)}(f) + \int_{U_{R_n}}  (f(x)\vee f(y))^{p-2} \eta_{n}^2(y) (f(x) - f(y))^2 dJ(x,y),$$
with $U_{r} = B_{r} \times B_r \setminus d$. The assumption $\int^\infty r/\av{f1_{B_r}}_p^p dr = \infty$ implies
$$\sum_{n = 1}^\infty \frac{R_n^2}{v_n} = \infty.$$
Once we show  $Q_1 = 0$, the assertion follows from Lemma~\ref{lemma:irreducible constant} after letting $R \to \infty$. 

Lemma~\ref{lemma:additional inequality} (applied to $r = R_{n-1}+s$ and $R = R_{n}+s$)  and $\eta_{n-1} \leq \eta_{n}$ yield 
\begin{align*}
 Q_{n-1}Q_n \leq Q_n^2 \leq \frac{C v_n}{(R_n - R_{n-1} - 2s)^2} \left(Q_n -  Q_{n-1}\right) \leq \frac{16 C v_n}{R_n^2}\left(Q_n -  Q_{n-1}\right).
\end{align*}
For the last inequality we used $R \geq 4s$ and $R_n = 2^nR$. If $Q_1 > 0$, this would imply 
$$\frac{R_n^2}{v_n}  \leq 16C \left( \frac{1}{Q_{n-1}}  - \frac{1}{Q_{n}}\right)$$
and hence 
$$\sum_{n = 2}^\infty \frac{R_n^2}{v_n} \leq  \frac{16C}{Q_{1}} < \infty,$$
a contradiction. 
\end{proof}

\subsection{Proof of the growth test for recurrence}
In this subsection we discuss how our version of Karp's theorem can be used to deduce  the volume growth test for recurrence Theorem~\ref{c:recurrence}.  As in the previous sections we assume that $\cE$ is a regular Dirichlet form without killing.

In order to prove the theorem we employ some abstract results from \cite{Kaj}, to which we also refer to for the following facts about extended Dirichlet spaces.  Recall that the extended Dirichlet space $D(\Ee)$ of $\cE$ is defined by
$$D(\Ee) = \{f \in L^0(m) \mid \text{ex. $\cE$-Cauchy sequence } (f_n) \text{ in }D(\cE) \text{ s.t. } f_n \to f\, m\text{-a.e.}\}.$$
A sequence $(f_n)$ as in the definition above is called {\em approximating sequence} for $f$. Let $f \in D(\Ee)$ and let  $(f_n)$ be an approximating sequence. We define
$$\Ee(f) = \lim_{n \to \infty} \cE(f_n).$$
This is indeed independent of the choice of the sequence $(f_n)$ such that $\Ee$ is an extension of $\cE$. Moreover, it holds that $D(\cE) = D(\Ee) \cap L^2(m)$. We need the following observation, which shows that our extension of the form via the measure-valued quadratic form $\Gamma$ is compatible with the extended Dirichlet space.

\begin{lemma}\label{lemma:compatibility}
 Let $f \in D(\Ee) \cap L^\infty(m)$. Then $f \in \dEl^*$ and 
 $$\Ee(f) = \int_X d\Gamma(f). $$
\end{lemma}
\begin{proof}
 Let $f \in D(\Ee) \cap L^\infty(m)$. We first show $f \in \dEl$. For a compact $K \subseteq X$ we choose $\varphi \in  D(\cE) \cap C_c(X)$ with $0 \leq \varphi  \leq 1$ and $\varphi = 1$ on $K$. According to \cite[Exercise~1.4.1]{FOT} such a function always exists. Then $\varphi f = f$ on $K$ and $\varphi f \in L^2(m)$. Since $D(\Ee) \cap L^\infty(m)$ is an algebra and $D(\Ee) \cap L^2(m) = D(\cE)$, we obtain $\varphi f \in D(\cE)$. This shows  $f \in \dEl$.

 Now let $(f_n)$ be an approximating sequence for $f$. Without loss of generality $\av{f_n}_\infty \leq \av{f}_\infty$ and according to \cite[Theorem~2.1.7]{FOT} we can additionally assume $f_n \to f$ q.e. (recall that we always choose quasi-continuous representatives). Since the jump measure does not charge sets of capacity $0$, Fatou's lemma yields 
 \begin{align*}
  \int_{X \times X \setminus d} (f(x) - f(y))^2 dJ(x,y) &\leq \liminf_{n \to \infty}\int_{X \times X \setminus d} (f_n(x) - f_n(y))^2 dJ(x,y)\\
  &\leq \liminf_{n \to \infty} \cE(f_n) \\
  &= \Ee(f) < \infty. 
 \end{align*}
 In particular, this implies $f \in \dEl^*$.

 We will show $\int_X d\Gamma(f_n) \to \int_X d\Gamma(f)$ as this yields
 $$\int_X d\Gamma(f) = \lim_{n \to \infty} \int_X d\Gamma(f_n) = \lim_{n \to \infty}\cE(f_n) = \Ee(f).$$

 According to \cite[Theorem~3.1]{Schmidt20}, for $\varphi \in D(\cE) \cap C_c(X)$ with $0 \leq \varphi \leq 1$ and $g \in D(\Ee) \cap L^\infty(m)$ we have (using  $\varphi g, \varphi g^2 \in D(\cE) \cap L^\infty(m)$, see first part of the proof)
 $$\cE(\varphi g) - \cE(\varphi g^2,\varphi) \leq \Ee(g).$$ 
 Using the Leibniz rule for $\Gamma^{(c)}$ and the discrete Leibniz rule for the integration with respect to $J$ we   obtain 
 $$ \int_X \varphi^2 d\Gamma^{(c)}(g) +   \int_{X \times X \setminus d} \varphi(x)\varphi(y) (g(x) - g(y))^2 dJ(x,y) = \cE(\varphi g) - \cE(\varphi g^2,\varphi).$$
 Letting $\varphi \nearrow 1$ this implies 
 $$\int_X d\Gamma(g)  \leq \Ee(g).$$
 In particular, since the LHS is a quadratic form in $g$, for $(g_n)$ in $D(\Ee) \cap L^\infty(m)$ the convergence $g_n \to g$ with respect to $\Ee$ implies $\int_X d\Gamma(g_n) \to \int_X d\Gamma(g)$. 
 
 Since $(f_n)$ is an approximating sequence for $f$, the lower semicontinuity of $\Ee$ with respect to $m$-a.e. convergence, see e.g. \cite[Lemma~2.3]{Schmidt20}, yields 
 $$\Ee(f - f_n) \leq \liminf_{m \to \infty} \Ee(f_m - f_n) =  \liminf_{m \to \infty} \cE(f_m - f_n) \to 0, \text{ as } n \to \infty. $$
 Hence, we obtain the claimed convergence of $\Gamma(f_n)$ to $\Gamma(f)$.
\end{proof}
 \begin{proof}[Proof of the volume growth test for recurrence, Theorem~\ref{c:recurrence}]
  According to \cite[Theorem~1 and Proposition~3]{Kaj} it suffices to show that any non-negative $h \in D(\Ee) \cap L^\infty(m)$ with $\Ee(h,\varphi) \geq 0$ for all nonnegative $\varphi \in D(\cE)$ is constant (such functions are called excessive). By Lemma~\ref{lemma:compatibility} such a function $h$ satisfies $h \in \dEl^*$ and 
  $$\int_X d\Gamma(h,\varphi) \geq 0$$
  for all non-negative $\varphi \in D(\cE)_c$. Now assume without loss of generality $\av{h}_\infty \leq 1$. Since $\Gamma(1) = 0$, the function $f = 1-h$ is nonnegative and subharmonic in our sense. Moreover, $\av{f1_{B_{r}}}_p^p \leq \av{f}_\infty^p m(B_r).$ Hence, the assumption on $m(B_r)$ and Theorem~\ref{t:Karp} yield that $f$ is constant. 
 \end{proof}

\bibliography{Lqharmonic2}
\bibliographystyle{plain}

\end{document}